\documentclass[11pt, reqno]{amsart}
\usepackage[margin=1.2in]{geometry}
\usepackage{amsfonts}
\usepackage{amsthm,amsmath}
\usepackage{amsbsy}
\usepackage{bbm}
\usepackage{natbib}
\usepackage{amssymb} 
\usepackage{verbatim}
 \usepackage{color}
\usepackage{csquotes}
\usepackage{setspace}
\usepackage[mathscr]{eucal}
\usepackage{times}
\usepackage[pdfpagemode=UseNone,bookmarksopen=false,colorlinks=blue]{hyperref}
\hypersetup{ colorlinks=true, linkcolor=blue, citecolor=blue,
  filecolor=blue, urlcolor=blue}

\newlength{\myfboxsep}
\newlength{\mywidth}
\allowdisplaybreaks
\newtheorem{theorem}{Theorem}[section]
\newtheorem{lemma}{Lemma}[section]

\newtheorem{corollary}{Corollary}[section]

\newtheorem{remark}{Remark}[section]

\theoremstyle{definition}
\newtheorem{definition}{Definition}[section]
\newtheorem{example}{Example}[section]

\theoremstyle{remark}

\numberwithin{equation}{section}

\DeclareMathOperator{\lito}{o}
\DeclareMathOperator{\bigo}{O}

\newcommand{\gmu}{G_{\mu}}

\newcommand{\nustar}{{\mu_*^{\eta',0,\nu}}}
\newcommand{\nuplus}{{\mu_{\boxplus,V}^{\gamma,\sigma}}}
\newcommand{\nuplusn}{{\mu_{\boxplus,V}^{\gamma,\sigma_n}}}
\newcommand{\nuplusgn}{{\mu_{\boxplus,V}^{\gamma/n,\sigma/n}}}

\newcommand{\muminus}{{\mu_{\boxplus,C}^{\eta',0,\nu}}}

\title{Regular variation and free regular infinitely divisible laws}
\author{Arijit Chakrabarty, Sukrit Chakraborty and Rajat Subhra Hazra}
\address{ Indian Statistical Institute\\ 203, B.T. Road, Kolkata- 700108, India }
\email{arijit.isi@gmail.com, sukrit049@gmail.com, rajatmaths@gmail.com}

\begin{document}

\begin{abstract}
In this article  the relation between the tail behaviours of a free regular infinitely divisible probability measure and its L\'evy measure is studied. An important example of such a measure is the compound free Poisson distribution, which often occurs as a limiting spectral distribution of certain sequences of random matrices. We also describe a connection between an analogous classical result of \cite{embrechts1979subexponentiality} and our result using the Bercovici-Pata bijection.
\end{abstract}
\maketitle
\section{Introduction}

The limiting spectral distribution (LSD) of product of two or more random matrices is important in the field of random matrix theory. It arises naturally, for example, in study of  multivariate F-matrix (the product of mutually independent sample covariance matrix and the inverse of another sample covariance matrix). The limiting spectral distributions of F-matrices were studied in \cite{MR585695}, \cite{MR914942}.  In addition, products of random matrices arise in study of high dimensional time-series, for example, see \cite{MR2609495}, \cite{pan2012asymptotic}. For a history of the product of random matrices the reader is referred to \cite{bairandprod}.

The existence of a non-random LSD of the product of a sample covariance matrix and a non-negative definite Hermitian matrix, which are mutually independent, was given  explicitly in terms of the Stieltjes transform in \cite{MR1370408}. A stronger result in this direction is obtained using the moment method and truncation arguments in \cite{bairandprod}, by replacing the non-negative definite assumption by a Lindeberg type one, on the entries of the Hermitian matrices. When one considers Wishart matrices, a more explicit description of the LSD  can be given in terms of free probability; see \cite{merlevede2016peligrad}, \cite{chakrabarty2018note}.

It is well known in random matrix theory that the Marchenko-Pastur law (also called the free Poisson distribution) turns out to be the limiting spectral distribution of a sequence of Wishart random matrices ($W_N$). Suppose for each $N \geqslant 1$, $Y_N$ is an $N \times N$ independent random Hermitian matrix with LSD $\rho$. It can be shown that the expected empirical distribution of $W_NY_N$ converges to $m\boxtimes \rho$ as $N \to \infty$ where $\boxtimes$ denotes the free multiplicative convolution. It is not difficult to see that $\rho$ is compactly supported if and only if so is $m \boxtimes \rho$. Therefore it is  natural to ask whether there is any relation between the tail behaviour of $m \boxtimes \rho$ and $\rho$? In this paper, an affirmative answer is given to that question when $\rho$ has a power law tail decay. Thus, based on the LSD of  $Y_N$, one can describe the tail behaviour of that of  $W_NY_N$. In general, it is very hard to write down an explicit formula for the limit distribution. 

It is noteworthy that the probability measures of the form $m \boxtimes \rho$ are free regular probability measures (see \cite{ariz12}) which form a special subclass of free infinitely divisible distributions (also called the $\boxplus$-infinitely divisible distributions, see \cite{bercovici1993free}). The free cumulant transform of a free regular probability measure can be described through a  L\'evy-Khintchine representation. Interestingly, it turns out that $\rho$ is the L\'evy measure of $m \boxtimes \rho$.  Therefore it is natural to wonder whether there is any relation between the tail behaviours of a free regular probability measure and its L\'evy measure. 

In classical probability theory, a classically infinitely divisible probability measure $\mu$ also enjoys a L\'evy-Khintchine representation in terms of its L\'evy measure $\nu$.  In \cite{embrechts1979subexponentiality}, it was shown that for a positively supported classically infinitely divisible probability measure (a subordinator) $\mu$, the tails of $\mu$ and its L\'evy measure $\nu$ are asymptotically equivalent if and only if any one of  $\mu$ or $\nu$ is subexponential. In analogy to the classical case, it is natural to pose whether free subexponentiality characterizes the tail equivalence of a free infinitely divisible probability measure and its free L\'evy measure. But unfortunately the result can not be extended to the bigger class of free infinitely divisible probability measures. Since according to \cite{ariz12}, the correct analogue of the positively supported classically infinitely divisible probability measures are the free regular probability measures, in this paper, we provide a partial answer in Theorem~\ref{main theorem-2} by showing the tail equivalence of a free regular probability measure and its free L\'evy measure in presence of regular variation. Note that regularly varying measures are the most important subclass of both free and classical subexponential distributions (\cite{hazra1}). As an application of this result, the exact tail behaviour of the free multiplicative convolution of Marchenko-Pastur law with another regularly varying measure is derived in Corollary~\ref{theoremhere}. Besides, the connection of these results with the classical case is not a mere coincidence. From the famous result of Bercovici and Pata (\cite{ber:pata:biane}), it is known that classical and free infinitely divisible laws are in a one-to-one correspondence. It is shown in Corollary~\ref{classifreeconnec} that in the regularly varying set-up, the classical infinitely divisible law and its image under the Bercovici-Pata bijection are tail equivalent. The free multiplicative convolution of a measure with Wigner's semicircle law also appears naturally as limits of many random matrix models. It is shown in Corollary~\ref{corwigprod} that the tail behaviour turns out to be different from the one involving the Marchenko-Pastur law.  

In Section~\ref{sec:prelim} the basic notations and transforms used in free probability are introduced. Subsequently, the main results and their proofs are in Section \ref{subsec:mainresults}. Section~\ref{sec:cors} collects some corollaries arising out of the main results. The proofs depend heavily on relations between the transforms and regular variation. To keep the article self contained, the main result of \cite{hazra1} is quoted in the Appendix.

\section{Preliminaries and Main results}\label{sec:prelim}
\subsection{Notations and basic definitions: }
A real valued measurable function $f$ defined on non-negative real line is called \textit{regularly varying} (at infinity) with index $\alpha$ if for every $t>0$, $f\left(tx\right)/f\left(x\right)\rightarrow t^{\alpha}$ as $x \rightarrow \infty$. The function $f$ is said to be a slowly varying function (at infinity) if $\alpha = 0$. Throughout this paper, regular variation of a function will always be considered at infinity. A distribution function $F$  on $[0,\infty)$ has regularly varying tail of index $-\alpha$ if $\overline{F}(x) = 1 - F(x)$ is regularly varying of index $-\alpha$. Since $\overline{F}(x) \rightarrow 0$ as $x \rightarrow \infty$, it necessarily holds that $\alpha \geqslant 0$. For further details about regular variation see \cite{extremeresnick}. 
 A distribution $F$ on $[0,\infty)$ is called (classical) subexponential if $\overline{F^{(n)}}(x)\sim \overline F(x)$ as $x\to\infty$, for all $n\ge 0$. Here $F^{(n)}$ denotes the $n$-th (classical) convolution of $F$. Both regular variation and subexponentiality of a probability measure $\mu$ is defined thorough its distribution function $F$.

The real line and the complex plane will be denoted by $\mathbb{R} $ and $ \mathbb{C}$, respectively. The notations $\mathbb C^+$ and $\mathbb C^-$ are used for the upper and the lower halves of the complex plane, respectively, namely, $\mathbb C^+ = \{z\in\mathbb C: \Im z>0\}$ and $\mathbb C^- = - \mathbb C^+$, while $\mathbb{R}^+ := [0,\infty)$. For a complex number $z$, $\Re z$ and $\Im z$  denote its real and imaginary parts, respectively. Given positive numbers $\eta$, $\delta$ and $M$, let us define the following cone:
$$\Gamma_{\eta}=\{z\in \mathbb{C}^+: |\Re z|<\eta \Im z\} \text{ and } \Gamma_{\eta,M}=\{z\in\Gamma_{\eta}: |z|>M\}.$$
 Then we shall say that $f\left(z\right) \to l$ as $z$ goes to $\infty$ non-tangentially, abbreviated by ``n.t.'', if for any $\epsilon>0$ and $\eta>0$, there exists $M\equiv M\left(\eta,\epsilon\right)>0$, such that $|f\left(z\right) - l|< \epsilon$, whenever $z \in \Gamma_{\eta, M}$. This is same as saying that the convergence in $\mathbb C^+$ is uniform in each come $\Gamma_\eta$. The boundedness can be defined analogously.

We use the notations ``$f\left(z\right)\approx g\left(z\right)$'', ``$f\left(z\right)=\lito\left(g\left(z\right)\right)$'' and ``$f\left(z\right)=\bigo\left(g\left(z\right)\right)$ as $z\to\infty$ n.t.'' to mean, respectively, that ``$f\left(z\right)/g\left(z\right)$ converges to a non-zero finite limit'', ``${f\left(z\right)}/{g\left(z\right)}\rightarrow 0$'' and ``$f\left(z\right)/g\left(z\right)$ stays bounded as $z\to\infty$ n.t.'' If the limit is $1$ in the first case, we write $f\left(z\right) \sim g\left(z\right)$ as $z\to\infty$ n.t. For $f\left(z\right)=\lito\left(g\left(z\right)\right)$ as $z\to\infty$ n.t., we shall also use the notations $f\left(z\right)\ll g\left(z\right)$ and $g\left(z\right)\gg f\left(z\right)$ as $z\to\infty$ n.t.

Following \cite{bercovici1993free}, we recall that a non-commutative probability space ($\mathbf{A}, \phi$) is
said to be a $\mathcal{W^{\star}}$ -probability space if $\mathbf{A}$ is a non-commutative von Neumann algebra
and $\phi$ is a normal faithful trace. A family of unital von Neumann subalgebras
${\left(\mathbf{A}_i \right)}_{{i \in I}} \subset \mathbf{A}$ in a $\mathcal{W^{\star}} $ -probability space is called free if $\phi\left(a_{1} a_{2} \cdots a_{n} \right) = 0$
whenever $\phi\left(a_{j} \right) = 0, a_{j} \in A_{i_{j}}$ , and   $i_{1} \neq i_{2} \neq \cdots \neq i_{n}$ . A self-adjoint operator $X$ is affiliated with $\mathbf{A}$ if $f\left(X\right) \in \mathbf{A}$ for any bounded Borel function $f$ on $\mathbb{R}$.
In this case it is said that $X$ is a (non-commutative) random variable. For
a self-adjoint operator $X$ affiliated with $\mathbf{A}$, the distribution of $X$ is the unique
measure $\mu_{X}$ in $\mathcal{M}$ satisfying
\begin{align*}
\phi\left(f\left(X\right)\right) = \int_{\mathbb{R}} f\left(x\right) \mu_{X}\left(dx\right)
\end{align*}
for every Borel bounded function $f$ on $\mathbb{R}$. The self-adjoint operators $\{X_i\}_{1 \leqslant i \leqslant p}$, affiliated with a von Neumann algebra $\mathbf{A}$, are called free if and only if the algebras $\{f(X_i) : f \text{ is bounded measurable}\}_{1 \leqslant i \leqslant p}$ are free.

$\mathcal{M} $ and $\mathcal{M}_{+} $ are the set of probability measures supported on $\mathbb{R}$ and $\mathbb{R}^+$  respectively. By $\mathcal{M}_{p}$ we mean the set of probability measures on $[0, \infty)$ whose $p$-th moment is finite and do not have the $(p+1)$-th moment. The set $\mathcal{M}_{p,\alpha}$ will contain all probability measures in $\mathcal{M}_p$ with regularly varying tail index $-\alpha$ such that $p \leqslant \alpha \leqslant p+1$.

For a probability measure $\mu\in \mathcal{M}$, its Cauchy transform is defined as
$$G_{\mu}\left(z\right)=\int_{-\infty}^{\infty}\frac{1}{z-t}d\mu\left(t\right),\ \ \  z\in \mathbb{C}^+.$$
Note that $G_{\mu}$ maps $\mathbb{C}^+$ to $\mathbb{C}^{-}$. Set $F_{\mu}=1/\gmu$, which maps $\mathbb{C}^+$ to $\mathbb{C}^+$. 

The free cumulant transform ($C_{\mu}$) and the Voiculescu transform ($\phi_{\mu}$) of a probability measure $\mu$ are defined as
\begin{align}
C_{\mu}\left(z\right) = z \phi_{\mu}\left(\frac{1}{z}\right) = z F_{\mu}^{-1}\left(\frac{1}{z}\right) - 1, \nonumber
\end{align}
for $z$ in a domain $D_{\mu} = \{ z \in \mathbb{C}^- : 1/z \in \Gamma_{\eta,M} \}$ where $F_{\mu}^{-1}$ is defined;
The free additive convolution of two probability measures $\mu_{1} , \mu_{2}$ on $\mathbb{R}$ is defined
as the probability measure $\mu_{1} \boxplus \mu_{2}$ on $\mathbb{R}$ such that $\phi_{\mu_{1} \boxplus \mu_{2}} \left(z\right) = \phi_{\mu_{1}}\left(z\right) + \phi_{\mu_{2}}\left(z\right)$ or
equivalently
$C_{\mu_{1} \boxplus \mu_{2}} \left(z\right) = C_{\mu_{1}}\left(z\right) + C_{\mu_{2}}\left(z\right)$
for $z \in D_{\mu_{1}} \cap D_{\mu_{2}}$ . It turns out that $\mu_{1} \boxplus \mu_{2}$ is the distribution of the sum $X_{1} +X_{2}$ of
two free random variables $X_{1}$ and $X_{2}$ having distributions $\mu_{1}$ and $\mu_{2}$ respectively.
On the other hand, the free multiplicative operation $\boxtimes$ on $\mathcal{M}$ is defined as follows (see \cite{bercovici1993free}). Let $\mu_{1}$, $\mu_{2}$  be probability measures on $\mathbb{R}$, with $\mu_{1} \in \mathcal{M}_{+}$  and let $X_{1}$, $X_{2}$ be
free random variables such that $\mu_{X_{i}} = \mu_{i}$ . Since $\mu_{1}$ is supported on $\mathbb{R}^{+}$ , $X_{1}$ is a
positive self-adjoint operator and $\mu_{X_{1}^{1/2}}$ is uniquely determined by $\mu_{1}$. Hence the
distribution $\mu_{X_{1}^{1/2}X_{2}X_{1}^{1/2}}$ of the self-adjoint operator $X_{1}^{1/2}X_{2}X_{1}^{1/2}$ is determined by
$\mu_{1} $ and $\mu_{2}$. This measure is called the free multiplicative convolution of $\mu_{1} $ and $\mu_{2} $
and it is denoted by $\mu_{1} \boxtimes \mu_{2} $ . This operation on $\mathcal{M}_+$ is associative and commutative.

We recall (Theorem $1.3$ and Theorem $1.5$ of \cite{taylorseries}) the remainder terms in Laurent series expansion of Cauchy and Voiculescu transforms for probability measures $\mu$ with finite $p$ moments and summarize the following expressions from \cite{hazra1}:
\begin{align}
r_{G_\mu}(z) &= z^{p+1} \left(G_\mu(z) - \sum_{j=1}^{p+1}m_{j-1}(\mu)z^{-j}\right)\label{eq: rG defn}
\intertext{and}
r_{\phi_\mu}(z) &= z^{p-1} \left(\phi_\mu(z) - \sum_{j=0}^{p-1}\kappa_{j+1}(\mu)z^{-j}\right),\label{eq: rphi defn}
\end{align}
where $\{m_j\left(\mu\right):j\le p\}$ and $\{\kappa_j\left(\mu\right):j\le p\}$ denotes the moment and free cumulant sequences of the probability measure $\mu$, respectively. 


\subsection{Classical infinite divisibility and known results} \label{subsec: id}
A probability measure $\mu$ is called classically infinitely divisible, if for every $n\in\mathbb{N}$, there exists a probability measure $\mu_n$ such that $\mu = \mu_n * \mu_n * \cdots * \mu_n (n \text{ times})$, where $*$ is the classical convolution of probability measures. A detailed description about classical infinite divisibility can be found in \cite{sato1999ki}.  It is well known that a probability measure $\mu$ on $\mathbb{R}$ is classically infinitely divisible if and only if its classical cumulant transform $C_\mu^*\left(w\right) := \log\int_\mathbb{R} e^{iwx} d\mu\left(x\right)$ has the following L\'evy-Khintchine representation (see \cite{sato1999ki} or \cite{barndIDnote})
\begin{equation}\label{eq: LK id111}
C_\mu^*\left(w\right) = i\eta w - \frac{1}{2}aw^2 + \int_\mathbb{R}\left(e^{iwt}-1-iwt\textbf{1}_{[-1,1]}\left(t\right)\right)d\nu\left(t\right), \hspace{.5cm}  w\in\mathbb{R},
\end{equation}
where $\eta\in\mathbb{R}$, $a\geqslant0$ and $\nu$ is a L\'evy measure on $\mathbb{R}$, that is, $\int_\mathbb{R}\min\left(1,t^2\right)d\nu\left(t\right)<\infty$ and $\nu\left(\{0\}\right)=0$. If this representation exists, the triplet ($\eta,a,\nu$) is called the classical characteristic triplet of $\mu$ and the triplet is unique.

Another form of $C_\mu^*\left(w\right)$ is given by
\begin{equation}
C_\mu^*\left(w\right) = i\gamma w + \int_\mathbb{R}\left(e^{iwt}-1-\dfrac{iwt}{1+t^2}\right)\frac{1+t^2}{t^2}d\sigma\left(t\right), \hspace{.5cm}  w\in\mathbb{R}, \nonumber
\end{equation}
where $\gamma$ is a real constant and $\sigma$ is a finite measure on $\mathbb{R}$.

One has the following relationships between the two representations (see equations $(2.3)$ below definition $2.1$ in \cite{barndIDnote}):
\begin{align}
a &= \sigma(\{0\}), \nonumber \\  
d\nu(t) &= \frac{1+t^2}{t^2}\mathbf{1}_{\mathbb{R} \setminus \{0\}}d\sigma(t), \label{cfrel1}\\ 
\eta &= \gamma + \int_{\mathbb{R}}t\left(\mathbf{1}_{[-1,1]}(t) - \frac{1}{1+t^2}\right)d\nu(t)\label{cfrel2}.
\end{align}
 In general, when one does not have the Brownian component, it is easier to consider the Laplace transform (if exists) of the measure, for example, in the case of compound Poisson. In this situation let us recall the classical result which studies the tail equivalence of L\'evy measure and the infinitely divisible distribution. In \cite{embrechts1979subexponentiality} it was shown subexponentiality is a property  which makes an infinitely divisible measure and its L\'evy measure tail equivalent.
\begin{theorem}[\citealp{embrechts1979subexponentiality}] \label{thm: embrechts}
Let $\mu$ be a classical infinitely divisible probability measure on $[0,\infty)$. Suppose $\mu$ has the L\'evy-Khintchine representation of the form,
\begin{align*}
f(s) = \int_{0-}^{\infty}e^{-st}d\mu(t) = exp\left\lbrace-as - \int_{0}^{\infty}(1 - e^{-st})d\nu(t)\right\rbrace
\end{align*} 
where $\nu$ is a L\'evy measure satisfying $\int_0^{\infty}\min\{1,t\}d\nu(t) < \infty$. Then the following statements are equivalent:
\begin{enumerate}
 \item[(a)] $\mu$ is subexponential,
 \item[(b)] $\overline{\nu}$ is subexponential,
 \item[(c)] $\mu(x,\infty) \sim \nu(x,\infty)$ as $x\rightarrow\infty$.
\end{enumerate}
Here, the probability measure $\overline{\nu}$, supported on the interval $(1,\infty)$ is defined by $$\overline{\nu}(1, x) = \nu(1,x]/\nu(1,\infty)\,.$$
\end{theorem}
The remarkable feature of this result is that tail equivalence gives subexponentiality. In Section~\ref{subsec:mainresults} we will address the partial extension of this result in the free setting. 
\subsection{Free infinite divisibility and free regular probability measures}
Free infinitely divisible probability measures are defined in analogy with classical infinitely divisible probability measures. Infinitely divisible measures can also be described in terms of a representation through Voiculescu and free cumulant transforms. A probability measure $\mu$ is called \textit{free infinitely divisible}, if for every $n\in\mathbb N$, there exists a probability measure $\mu_n$ such that $\mu = \mu_n \boxplus \mu_n \boxplus \cdots \boxplus \mu_n (n \text{ times})$ holds. Also a probability measure $\mu$ on $\mathbb{R}$ is $\boxplus$-infinitely divisible i.e. free infinitely divisible if and only if there exists a finite measure $\sigma$ on $\mathbb{R}$ and a real constant $\gamma$, such that 
\begin{equation}\label{voiculescu}
\phi_{\mu}\left(z\right) = \gamma + \int_{\mathbb{R}}\frac{1+zt}{z-t}d\sigma\left(t\right), \hspace{1cm} z\in\mathbb{C}^+
\end{equation}

A probability measure $\mu$ on $\mathbb{R}$ is $\boxplus$-infinitely divisible if and only if the free cumulant transform has the representation:
\begin{equation}\label{lkr}
C_\mu^\boxplus(z) = \eta z + az^2 + \int_\mathbb{R}\Big(\frac{1}{1-zt}-1-tz\textbf{1}_{[-1,1]}(t)\Big)d\nu(t),  z\in\mathbb{C}^-,
\end{equation}
where $\eta\in\mathbb{R}$, $a\geqslant0$ and $\nu$ is called the L\'evy measure on $\mathbb{R}$. In the expressions \eqref{voiculescu} and \eqref{lkr}, similar to the equations \eqref{cfrel1} and \eqref{cfrel2} holds true (see \cite[ Proposition $4.16$]{barndIDnote}). The free characteristic triplet $(\eta,a,\nu)$ of a probability measure $\mu$ is unique.


For a free infinitely divisible probability measure $\mu$ on $\mathbb{R}$ where the L\'evy measure (Definition $2.1$ in \cite{barndIDnote}) $\nu$ satisfies $\int_{\mathbb{R}}\min\left(1,|t|\right)d\nu\left(t\right)<\infty$ and $a=0$ the L\'evy-Khintchine representation \eqref{lkr} reduces to 
\begin{equation} \label{cumulanttt}
C_\mu^\boxplus\left(z\right) = \eta'z + \int_\mathbb{R}\left(\dfrac{1}{1-zt}-1\right)d\nu\left(t\right), \, \,\,  z \in \mathbb{C}^-,
\end{equation}
where $\eta'\in\mathbb{R}$. The measure $\mu$ is called a \textbf{free regular infinitely divisible distribution} (or \textbf{regular $\boxplus$-infinitely divisible measure}) if $\eta'\geqslant0$ and $\nu((-\infty,0])=0$.

The most typical example is compound free Poisson distributions. If the drift term $\eta^{\prime}$ is zero and the L\'evy measure $\nu$ is $\lambda\rho$ for some constant $\lambda > 0$ and a probability measure $\rho$ on $\mathbb{R}$, then we call $\mu$ a compound free Poisson distribution with rate $\lambda$ and jump distribution $\rho$. To clarify these parameters, we denote $\mu = \pi\left(\lambda, \rho\right)$.
\begin{example}[\citealp{ariz12}, Remark~8]\label{cexamp}\
\begin{enumerate}
\item The Marchenko-Pastur law $m$ is a compound free Poisson with rate 1 and
jump distribution $\delta_{1}$ .
\item \label{fat:2} The compound free Poisson $\pi\left(1, \rho\right)$ coincides with the free multiplication $m \boxtimes \rho$.
\end{enumerate}
\end{example}

We shall use both the Voiculescu transform and the cumulant transform to state our theorems. The notations $\nuplus$ or $\muminus$ shall occur whenever we write the Voiculescu transform or the cumulant transform of a free regular probability measure $\mu$ respectively. The indices $V$ and $C$ are used to distinguish between the occurrence in Voiculescu transform or in the cumulant transform. The use of $\gamma$, $\sigma$, $\eta'$ and $\nu$ in the indices are clear from \eqref{bercovici} and \eqref{bercumulant} while in $\muminus$, the index $0$ is to indicate the non existence of the Gaussian part in the representation of the cumulant transform. Let $\nuplus=\muminus$ be a free regular infinitely divisible probability measure. Then its Voiculescu and cumulant transforms have the representations:
\begin{align} \label{bercovici}
\phi_\nuplus\left(z\right) &= \gamma + \int_{\mathbb{R}^+} \frac{1+tz}{z-t} d \sigma\left(t\right),\\
C_{\muminus}\left(z\right) &= \eta'z + \int_{\mathbb{R}^+}\left(\dfrac{1}{1-zt}-1\right)d\nu\left(t\right)\label{bercumulant}
\end{align}
respectively following \eqref{voiculescu} and \eqref{cumulanttt}. In the above representation the pair $(\gamma, \sigma)$ is related to $(\eta^\prime, \nu)$ in the following way:
\begin{align} \label{ravan}
\begin{split}
d\sigma(t) &= \frac{t^2}{1+t^2}d\nu(t),  \\
\gamma &= \eta^{\prime} + \int_{\mathbb{R}^+}\frac{t}{1+t^2}d\nu(t),  \hspace{.7cm}\eta^{\prime} \geqslant 0.
\end{split}
\end{align}
The proof of Theorem \ref{main theorem-2} demands the finiteness of the measure $\sigma$ appearing the Voiculescu transform while the L\'evy measure may not be a finite measure for a free regular infinitely divisible measure. 

\section{Main results and their proofs}\label{subsec:mainresults}
Now we are ready to state the main results of the paper while keeping in mind all the notations defined above. The following theorem gives us the tail equivalence between a free regular probability measure and the finite measure $\sigma$ occurring in the Voiculescu transform.
\begin{theorem}\label{main theorem-2}
Suppose that $\nuplus$ is free regular infinitely divisible measure. Assume that either $\nuplus$ or $\sigma$ is concentrated on $[0,\infty)$. Then the following statements are equivalent:
\begin{enumerate}
\item $\nuplus$ has regularly varying tail of index $-\alpha$.
\item $\sigma$ has regularly varying tail of index $-\alpha$.
\end{enumerate}
If either of the above holds, then $\nuplus\left(x,\infty\right) \sim \sigma\left(x,\infty\right)$ as $x\rightarrow\infty$.
\end{theorem}
We fix the notations $m_{-1}(\sigma) = \gamma$, $m_{0}(\sigma) = \sigma(\mathbb{R}^+)$ and $\overline\sigma$ for the probability measure $\sigma/m_0\left(\sigma\right)$. Recall the remainder terms of the Cauchy and Voiculescu transforms as defined in \eqref{eq: rG defn} and \eqref{eq: rphi defn} respectively. To prove Theorem \ref{main theorem-2}, we first state and prove the following Lemma.
\begin{lemma} \label{cor: id}
Let $\nuplus$ be a regular $\boxplus$-infinitely divisible probability measure.
\begin{enumerate}
\item Voiculescu transform of $\nuplus$ and Cauchy transform of $\overline\sigma$ are related by \label{cor: id transf}
    \begin{equation} \label{eq: id transf}
      \phi_\nuplus\left(z\right) = m_{-1}\left(\sigma\right) - m_0\left(\sigma\right)z + \left(1+z^2\right) m_0\left(\sigma\right) G_{\overline\sigma}\left(z\right).
    \end{equation}
\item 
If either $\nuplus$ or $\sigma$ has its support concentrated on $[0, \infty)$, then so does the other. \label{cor: id suppextra}
Further, in this case, $\nuplus$ and $\sigma$ have same number of moments. 
\item If both $\nuplus$ and $\sigma$ have $p$ moments, then the $p$ cumulants of $\nuplus$ and $p$ moments of $\sigma$ satisfy the relation
\begin{equation}\label{cumulants-id}
\kappa_p\left(\nuplus\right)=m_{p-2}\left(\sigma\right)+m_p\left(\sigma\right)
\end{equation}
and the remainder terms of $\phi_{\nuplus}$ and $G_{\overline\sigma}$ satisfy
    \begin{equation} \label{eq: id remainder}
    r_{\phi_\nuplus}\left(z\right) = m_{p-1}\left(\sigma\right)z^{-1} + m_p\left(\sigma\right) z^{-2} + \left(1+z^{-2}\right) m_0\left(\sigma\right) r_{G_{\overline\sigma}}\left(z\right).
 \end{equation}
\end{enumerate}
\end{lemma}
\begin{proof}
\begin{enumerate}
\item Using~\eqref{bercovici}, we have
\begin{align*}
\phi_{\nuplus}\left(z\right)&=\gamma+m_0\left(\sigma\right) \int_{0}^{\infty}\frac{1+tz}{z-t}d\overline\sigma\left(t\right)\\
&=\gamma+m_0\left(\sigma\right) G_{\overline\sigma}\left(z\right)+m_0\left(\sigma\right) z\int_{0}^{\infty}\frac{t}{z-t}d\overline\sigma\left(t\right)\\
&=m_{-1}\left(\sigma\right)+m_0\left(\sigma\right) G_{\overline\sigma}\left(z\right)-m_0\left(\sigma\right)z+m_0\left(\sigma\right) z^2G_{\overline\sigma}\left(z\right).
\end{align*}

\item Suppose $\nuplus $ is concentrated on the positive axis. Since  $\nuplus= \muminus$, the definition of free regularity ensures that the support of $\nu$ is contained in $[0,\infty)$. From~\eqref{ravan} we have $\sigma$ is concentrated on $[0,\infty)$.  Conversely, let $\sigma$ be supported on $[0,\infty)$. Since $\eta^\prime\ge 0$ it is follows from~\eqref{ravan} that $\int_0^1 \frac{1}{t} d\sigma (t) \leqslant \gamma$. By Lemma 9 of \citet{bena-surprize} it follows that $\nuplus$ is supported on $[0,\infty)$.

Here it is easy to conclude that for a non trivial free regular probability measure $\nuplus$, one must have $\gamma > 0$. Now we shall prove the existence of a $p$ moment of $\sigma$ is equivalent to the existence of a $p$ moment of $\nuplus$.  We shall follow the proof of Proposition 2.3 of \citet{taylorseries}.

 First suppose $\sigma$ admits a moment of order $p$. For all positive integer $n$, let us define the positive finite measure $\sigma_n$ on $[0,\infty)$ by $\sigma_n(A) = \sigma \big(A \cap [0,n)\big)$. By dominated convergence theorem $\sigma_n$ converges weakly to $\sigma$. Thus by \citet[Theorem $3.8$]{barndorff2002} we have $\nuplusn$ converges weakly to $\nuplus$. Therefore, $$\int_0^{\infty}t^p d\nuplus(t) \leqslant \liminf_n \int_0^{\infty}t^p d\nuplusn(t).$$
The range of the integral is $\mathbb{R}^+$ instead of $\mathbb{R}$ because $\nuplusn$ is again a free regular measure (since $\nuplus$ is so) and the first part of Lemma \ref{cor: id}(\ref{cor: id suppextra}) gives $\nuplusn$ is concentrated on $[0,\infty)$. Thus, $$\int_0^{\infty}t^p d\nuplus(t) \leqslant \liminf_n m_p(\nuplusn).$$
To show that $\nuplus$ has $p^{th}$ moment finite, it is enough to show that the sequence $\{m_p(\nuplusn)\}_n$ is bounded. By \citet[equation $(2.1)$]{taylorseries}, we have the $q^{th}$ free cumulant $\kappa_q(\nuplusn) = m_{q-2}(\sigma_n)+ m_q(\sigma_n)$ since $\sigma_n$'s are compactly supported (with the convention that $m_{-1}(\sigma_n)= \gamma$). So, for all $n$,
\begin{align*}
m_p(\nuplus) &= \sum_{\pi \in NC(p)} \prod_{V \in \pi} \kappa_{|V|}(\nuplusn)\\
&= \sum_{\pi \in NC(p)} \prod_{V \in \pi}  \big(m_{q-2}(\sigma_n)+ m_q(\sigma_n)\big)\\
&\leqslant \sum_{\pi \in NC(p)} \prod_{V \in \pi}  \big(m_{q-2}(\sigma)+ m_q(\sigma)\big) < \infty,
\end{align*}
where $NC(p)$ is the set of all non crossing partitions of $\{1,2, \ldots, n\}$ and $|V|$ is the number of elements in the block $V$ of $\pi$.

Next suppose $\nuplus$ admits a moment of order $p$. Then by \citet[Theorem $1.3$]{taylorseries}, $\phi_\nuplusgn$ admits a Lauent series expansion of order $p+1$. Thus for all positive integer $n$, we have $\phi_\nuplusgn(z) = \frac{1}{n}\phi_\nuplus(z)$. Now support of $\nuplusgn(z)$ is contained in $[0,\infty)$ (as $\nuplusgn(z)$ is a free regular measure with $\sigma_n$ has support on $[0,\infty)$) and the uniqueness of the Laurent series expansion allows us to conclude that $\nuplusgn(z)$ has a Laurent series expansion of order $p+1$. Moreover we have $\kappa_i(\nuplusgn) = \frac{1}{n}\kappa_i(\nuplus)$ for all $i \in \{1,2, \ldots, p\}$. From \citet[Theorem $5.10(iii)$]{bercovici1993free}, we conclude that \begin{equation*}
d\sigma(t) = \lim_{n \to \infty} \frac{nt^2}{1+t^2}d\nuplusgn(t).
\end{equation*}
Therefore we have,
\begin{align*}
\int_0^{\infty} t^p d\sigma(t) &\leqslant \liminf_n \int_0^{\infty} \frac{t^pnt^2}{1+t^2}d\nuplusgn(t)\\
&\leqslant \liminf_n \int_0^{\infty} nt^pd\nuplusgn(t)\\
&= \liminf_n n m_p(\nuplusgn)\\
&= \liminf_n \sum_{\pi \in NC(p)} n \prod_{V \in \pi} \kappa_{|V|}(\nuplusgn)\\
&= \liminf_n \sum_{\pi \in NC(p)} n^{1-\#\pi} \prod_{V \in \pi} \kappa_{|V|}(\nuplus) < \infty,
\end{align*}
where in the third line we have used $\nuplusgn((-\infty, 0)) = 0$ since for all $n$, $\nuplusgn$ is free regular and $\sigma/n$ has support on $[0,\infty)$ and in the last line $\#\pi$ indicated the number of blocks in the partition $\pi$.

\item If both $\nuplus$ and $\sigma$ have $p$ moments finite, considering Laurent series expansion of $G_{\overline\sigma}$ in~\eqref{eq: id transf} and the fact that $m_j\left(\sigma\right) = m_0\left(\sigma\right) m_j\left(\overline\sigma\right)$ for $0\le j\le p$, we have
\begin{align*}
\phi_{\nuplus}(z)=&m_{-1}(\sigma)-m_0(\sigma)z+(1+z^2)\sum_{j=1}^{p+1}m_{j-1}(\sigma)z^{-j}\\
&\qquad+(1+z^2)z^{-(p+1)}m_0(\sigma) r_{G_{\overline\sigma}}(z)\\
=&\sum_{j=1}^{p} (m_{j-2}(\sigma)+m_{j}(\sigma))z^{-(j-1)} +z^{-(p-1)} \big( m_{p-1}(\sigma)z^{-1}\\
&\qquad+m_{p}(\sigma)z^{-2}+(1+z^{-2})m_0(\sigma) r_{G_{\overline\sigma}}(z) \big).
\end{align*}
Since $r_{G_\sigma}\left(z\right)=\lito\left(1\right)$ as $z\to\infty$ n.t., we have
$$m_{p-1}\left(\sigma\right)z^{-1}+m_{p}\left(\sigma\right)z^{-2}+\left(1+z^{-2}\right)m_0\left(\sigma\right) r_{G_{\overline\sigma}}\left(z\right) = \lito\left(1\right)$$ as $z\to\infty$ n.t. Thus, by uniqueness of Laurent series expansion (which is equivalent to the uniqueness of Taylor series expansion given in \citet[Lemma~A$.1$]{taylorseries}), we obtain~\eqref{cumulants-id} as well as~\eqref{eq: id remainder}.
\end{enumerate} 
\end{proof}
It can be shown, using the expansions of Voiculescu and Cauchy transforms, that, if $\nuplus$ is a compactly supported probability measure, then $\sigma$ is also compactly supported and their cumulants and moments are related exactly by the formula stated in \eqref{cumulants-id}. Further note that $m_{p-2}\left(\sigma\right) + m_p\left(\sigma\right)$ is also the classical cumulant of a classical infinitely divisible distribution.

It is obvious that Lemma~\ref{cor: id}\eqref{cor: id suppextra} shows the assumptions on the supports of $\nuplus$ or $\sigma$ in Theorem~\ref{main theorem-2} are actually equivalent.

\begin{proof}[Proof of Theorem \ref{main theorem-2}]
First assume that $\nuplus$ is regularly varying with tail index $-\alpha$ for some $\alpha \geqslant 0$. Then there exists a unique nonnegative integer $p$ such that $\alpha\in[p,p+1]$ and the measure $\nuplus\in\mathcal M_{p,\alpha}$. Also, by Lemma~\ref{cor: id}\eqref{cor: id suppextra}, we have $\overline\sigma\in \mathcal M_p$ as well. Furthermore evaluating~\eqref{eq: id remainder} at $z=iy$ and equating the real and the imaginary parts respectively, we have,
\begin{align}
\left(1-y^{-2}\right) m_0\left(\sigma\right) \Re r_{G_{\overline\sigma}}\left(iy\right) - m_p\left(\sigma\right) y^{-2} &= \Re r_{\phi_{\nuplus}}\left(iy\right) \label{eq: Levy real}
\intertext{and}
\left(1-y^{-2}\right) m_0\left(\sigma\right) \Im r_{G_{\overline\sigma}}\left(iy\right) - m_{p-1}\left(\sigma\right)y^{-1} &= \Im r_{\phi_{\nuplus}}\left(iy\right). \label{eq: Levy imaginary}
\end{align}
Now if $\alpha\in[p,p+1)$, using Theorem~\ref{thm: error equiv}, we have from~\eqref{eq: Levy imaginary}, as $y\to\infty$,
\begin{align*}
\left(1-y^{-2}\right) m_0\left(\sigma\right) \Im r_{G_{\overline\sigma}}\left(iy\right) - m_{p-1}\left(\sigma\right)y^{-1} &= \Im r_{\phi_{\nuplus}}\left(iy\right) \\
&\sim -\frac{\frac{\pi\left(p+1-\alpha\right)}2}{\cos\frac{\pi\left(\alpha-p\right)}2} y^p\nuplus\left(y,\infty\right),
\end{align*}
which is regularly varying of index $-\left(\alpha-p\right)$ with $\alpha-p<1$. Thus, as $y\to\infty$,
\begin{align*}
m_0\left(\sigma\right) \Im r_{G_{\overline\sigma}}\left(iy\right) &\sim\left(1-y^{-2}\right)m_0\left(\sigma\right) \Im r_{G_{\overline\sigma}}\left(iy\right)\\
&\sim -\frac{\frac{\pi\left(p+1-\alpha\right)}2}{\cos\frac{\pi\left(\alpha-p\right)}2} y^p\nuplus\left(y,\infty\right)
\end{align*}
 and is also regularly varying of index $-\left(\alpha-p\right)$ and again by Theorem~\ref{thm: error equiv}, $\overline\sigma$ and hence $\sigma$ has regularly varying tail of index $-\alpha$ and
$$\Im r_{G_{\overline\sigma}}\left(iy\right) \sim -\frac{\frac{\pi\left(p+1-\alpha\right)}2}{\cos\frac{\pi\left(\alpha-p\right)}2} y^p\overline\sigma\left(y,\infty\right) \text{ as $y\to\infty$.}$$
Putting two asymptotic equivalences together, we get $\nuplus\left(y,\infty\right)\sim m_0\left(\sigma\right)\overline\sigma\left(y,\infty\right)=\sigma\left(y,\infty\right)$ as same argument works for the case $\alpha=p+1$ with the help of Theorem~\ref{thm: error equiv new} and equation~\eqref{eq: Levy real}.

To get the converse statement, we shall start with $\sigma$ to be regularly varying with index $-\alpha$. Thus $\sigma \in \mathcal M_{p,\alpha}$ for some integer $p \geqslant 0$. Lemma~\ref{cor: id}\eqref{cor: id suppextra} gives $\nuplus \in \mathcal M_p$ also, and we get the equations \eqref{eq: Levy real} and \eqref{eq: Levy imaginary}. Arguing exactly the same way like above we shall be able to conclude that $\nuplus$ is regularly varying with tail index $-\alpha$.   
\end{proof}

Noting the relations between the measures appearing in the L\'evy-Khintchine representations of the Voiculescu and cumulant transform of a free regular measure, the following corollary is immediate and this will also be very important to link our result with the classical one in Corollary \ref{classifreeconnec}.
\begin{corollary} \label{sita}
Suppose $\muminus$ is a free regular infinitely divisible measure. Then the following are equivalent:
\begin{enumerate}
\item $\muminus$ has regularly varying tail of index $-\alpha$.
\item $\nu$ has regularly varying tail of index $-\alpha$.
\end{enumerate}
If either of the above holds, then $\muminus\left(x,\infty\right) \sim \nu\left(x,\infty\right)$ as $x\rightarrow\infty$.
\end{corollary}
\begin{remark}
Note that in Corollary \ref{sita}, the measure $\nu$ may not be a finite measure. Since $\nu$ being a L\'evy measure of a free regular probability measure we have $\nu(1,\infty) < \infty$ and therefore there is no ambiguity in talking about its tail behaviour.
\end{remark}
\begin{proof}[Proof of Corollary \ref{sita}]
First we observe the following with the notations $\sigma$, $\nu$ and $a$ be as in~\eqref{ravan}. Suppose $a=0$. Then from ~\eqref{ravan}, taking integral from $x$ to infinity on both sides we get,
\begin{align*}
\nu(x,\infty) &= \sigma(x,\infty) + \int_{x}^{\infty}\frac{1}{t^2}d\sigma(t) \\
&\leqslant (1 + \frac{1}{x^2})\sigma(x,\infty). \text{  since  } x < t.
\end{align*}
Therefore,
\begin{align*}
\sigma(x,\infty) \leqslant \nu(x,\infty) \leqslant (1 + \frac{1}{x^2})\sigma(x,\infty).
\end{align*}
Taking limit as $x \rightarrow \infty$ we get
\begin{equation}\label{ram1.1}
\sigma\left(x,\infty\right) \sim \nu\left(x,\infty\right) \text{  as  } x\rightarrow \infty.
\end{equation}
Now Corollary \ref{sita} is immediate from Theorem~\ref{main theorem-2} and the equation~\eqref{ram1.1} as 
\begin{align*}
\muminus\left(x,\infty\right) = \nuplus\left(x,\infty\right) \overset{Theorem~\ref{main theorem-2}}\sim \sigma\left(x,\infty\right) \overset{\eqref{ram1.1}}\sim \nu\left(x,\infty\right).                     
\end{align*}
\end{proof}

\section{Some corollaries}  \label{sec:cors}
As an application of our main result we study the compound free Poisson distribution which turn out to be the free analogue of the classical compound Poisson distribution. Recall, that if $G$ is a proper distribution on $[0,\infty)$ and $\lambda>0$ then the (classical) compound Poission distribution is defined as 
$$F(x)= e^{-\lambda}\sum_{n=0}^\infty \frac{\lambda^n}{n!} G^{(n)}(x)$$
where $G^{(0)}$ is dirac mass at 0 and $G^{(n)}$ is the $n$-th classical convolution of $G$. It was shown in \cite[Theorem 3]{embrechts1979subexponentiality} 
that $F$ is subexponential if and only if $G$ is subexponential and this is also equivalent to $\overline F(x)\sim \lambda \overline G(x)$ as $x\to\infty$. We show that a partial analogue of this result is true in the free setting when one restricts to regularly varying measures.

The representation of a compound free Poisson distribution $\mu = \pi\left(1, \rho\right)$ as $\mu = m \boxtimes \rho$ makes it an interesting object to study further as they arise as limits of empirical distribution of random matrices.

As a corollary of Corollary \ref{sita}, we get the following:
\begin{corollary} \label{theoremhere}
Let $\rho$ be a positively supported probability measure. Then for the compound free Poisson distribution $\mu = \pi\left(1, \rho\right)$ which coincides with the free multiplication $m \boxtimes \rho$, the following are equivalent.
\begin{enumerate}
\item The tail of $\mu$ is regularly varying with index $-\alpha$.
\item The tail of $\rho$ is regularly varying with index $-\alpha$.
\end{enumerate}
If any of the above holds, then $\mu\left(y,\infty\right) \sim \rho\left(y,\infty\right)$ as $y \rightarrow \infty$.
\end{corollary}
\begin{proof}[Proof of Corollary \ref{theoremhere}]
This result follows directly from Corollary \ref{sita} while noticing from the example \ref{cexamp} that $\rho$ is the L\'evy measure of the compound free Poisson distribution $\mu = m \boxtimes \rho$. 
\end{proof}

Now we describe two situations where the above results can be applied. The first one is for random matrices while the other one is for the free stable laws.
\begin{example}
As mentioned in the introduction, $m \boxtimes \rho$ often occurs as a limiting spectral distribution. For example consider for all $N \geqslant 1$, $W_N = \frac{1}{M_N}X_N^{*}X_N$ where $X_N$ is a complex Gaussian random matrix with i.i.d. entries and the sequence $\{M_N\}_{N \geqslant 1}$ is such that $\lim_{N \to \infty}N/M_N = \lambda \in (0,\infty)$.  Also take $Y_N$,  for all $N \geqslant 1$ to be random complex Hermitian matrices independent of the entries of $X_N$. Suppose there exists a non random probability measure $\rho$ on $\mathbb{R}$ such that empirical spectral distribution of $Y_N$ converges to $\rho$ weakly in probability. In this setup when $\lambda = 1$, \citet[Theorem $2.3$]{chakrabarty2018note} tells us that the expected empirical spectral distribution of $W_NY_N$ converges to $m \boxtimes \rho$ weakly as $N \to \infty$. Therefore if we take $\rho$ to be regularly varying with tail index $-\alpha$, $\alpha \geqslant 0$, we are able to conclude that the tail of the limiting spectral distribution of $W_NY_N$ is same as that of $\rho$ using Corollary \ref{theoremhere}. \qed
\end{example}
\begin{example}\label{stableexample}
Following \cite{ber:pata:biane} we define two probability measures $\mu$ and $\nu$ to be equivalent (denote as $\mu \sim \nu$) if $\mu(S) = \nu(aS+b)$ for every Borel set $S \subseteq \mathbb{R}$, for some $a \in \mathbb{R}^+$ and $b \in \mathbb{R}$. A measure $\mu$ (excluding point mass measures) is said to be $\boxplus$-stable if for every $\nu_1, \nu_2 \in \mathcal{M}$ such that $\nu_1 \sim \mu \sim \nu_2$, it follows that $\nu_1 \boxplus \nu_2 \sim \mu$. Associated with every $\boxplus$-stable measure $\mu$ there is a number $\alpha \in (0,2]$ such that the measure $\mu \boxplus \mu$ is a translate of the measure $D_{1/{2^{\alpha}}}\mu $ where $D_a\mu(S) = \mu(aS)$. The number $\alpha$ is called the stability index of $\mu$. The probability measure $\mu^{(2)}$ will be the image of $\mu$ under the map $t \to t^2$ on $\mathbb{R}$.

We give a proper example where Corollary \ref{theoremhere} follows directly. From the appendix of \cite{ber:pata:biane} we get that the Voiculescu transform of a $\boxplus$-stable probability measure with stability index $\alpha \in (0,1)$ is of the form $$\phi(z) = -e^{i\alpha \rho \pi}z^{-\alpha + 1}$$ where $\rho$ is called the asymmetry coefficient. Now using Theorem  \ref{thm: error equiv} of Appendix we can conclude that the $\boxplus$-stable probability measures in $\mathcal{M}_0$ with stability index $\alpha \in (0,1)$ are exactly regularly varying probability measures with tail index $-\alpha$.

Let $\mu_{\alpha}$ be a regularly varying symmetric free $\alpha$-stable law with $0 < \alpha <2$. Then $\mu_{\alpha}^{(2)} = \rho_{\frac{\alpha}{2}} \boxtimes m$, where $\rho_{\frac{\alpha}{2}}$ is a free positive $\frac{\alpha}{2}$ stable law.  

The above statement can be verified by the following arguments. First from \citet[Corollary $21$]{fidmix2012abreu} observe  that the positive $\frac{\alpha}{2}$-stable law $\mu_{\alpha}^{(2)}$ enjoys the relation $$\mu_{\alpha}^{(2)} = ({\rho}_{\beta} \boxtimes {\rho}_{\beta})\boxtimes m,$$ where ${\rho}_{\beta}$ is a free positive $2\alpha/(2+\alpha)$ stable law. Applying \citet[Proposition $13$]{arizsymmetric}, it follows that ${\rho}_{\beta} \boxtimes {\rho}_{\beta} = \rho_{\frac{\alpha}{2}}$. Hence $\mu_{\alpha}^{(2)} = \rho_{\frac{\alpha}{2}} \boxtimes m$. Observe $\mu_{\alpha}^{(2)}$ and $\rho_{\frac{\alpha}{2}}$ are in $\mathcal{M}_0$ implies that both have regularly varying tail of index $-\frac{\alpha}{2}$. The Corollary~\ref{theoremhere} can be seen as generalizing this behaviour to a much more general class of probability measures. \qed
\end{example}


It is a pertinent question that whether the conclusion involving Marchenko-Pastur law can be replaced by the standard Wigner's semicircle law, $w$. The measures of the form $w \boxtimes \rho$ for some $\rho \in \mathcal{M}_+$ has appeared as the limiting spectral distributions of random matrices (see \citet{anderson2008law, hazralongrange, chakrabarty2018spectra}), free type $W$ distributions (see \citet{fidmix2012abreu}) and in several other places. 

The first observation in this regard is that in general one cannot say that $w \boxtimes \rho$ is free infinitely divisible for some $\rho \in \mathcal{M}_+$. In fact if one considers the measure $w_+$ having density $$f_{w_+}(x) = \frac{1}{2\pi}\sqrt{4-{(x-2)}^2}\mathbf{1}_{[0,4]}(x),$$ then it was shown in \citet[Corollary $3.5$]{sakumacounterex} that $w \boxtimes w_+$ is not a free infinitely divisible measure. The obstacle comes from the fact that $w_+$ is not free regular. So a valid question in this regard is whether $w \boxtimes \rho$ is regularly varying if $\rho$ is free regular with regularly varying tail? We give a partial answer when the measure $\rho \boxtimes \rho$ is regularly varying of index $-\alpha$, $\alpha \geqslant 0$. Such a particular case can arise in free stable laws and goes back to the works of \citet[Proposition A$4.3$]{ber:pata:biane} which states if $\rho_{\alpha}$ and $\rho_{\beta}$ are free stable laws of index $\alpha$, $\beta$ $\in (0,1)$ respectively, then $\rho_{\alpha} \boxtimes \rho_{\beta}$ is free stable law of index  $\frac{\alpha \beta}{\alpha +\beta - \alpha \beta}$ and hence regularly varying of index $-\frac{\alpha \beta}{\alpha +\beta - \alpha \beta}$ (as discussed in the second paragraph of Example \ref{stableexample}). So combining these observations we have the following corollary where we use the definition of regularly varying measures supported on $\mathbb{R}$ instead of $\mathbb{R}^+$. Since we restrict ourselves to symmetric probability measures we don't go into the details of the definition of regular variation of such tail balanced measures.

\begin{corollary} \label{corwigprod}
Let $\rho \in \mathcal{M}_+$ and $w$ be the standard Wigner measure. Then the following are equivalent.
\begin{enumerate}
\item \label{example1of1} $\rho_0 = \rho \boxtimes \rho$ is free regular infinitely divisible, regularly varying probability measure with tail index $-\alpha$, $\alpha \geqslant 0$.
\item \label{example1of2}$\mu = w \boxtimes \rho$ is free infinitely divisible, regularly varying with tail index $-\frac{\alpha}{2}$.
\end{enumerate}
\end{corollary}
\begin{proof} Assume \eqref{example1of1}. Theorem $22$ of \cite{fidmix2012abreu} says that for $\rho \in \mathcal{M}_+$ and $w$ be the standard Wigner measure, then $\rho_0 = \rho \boxtimes \rho$ is a free regular infinitely divisible probability measure if and only if $\mu = w \boxtimes \rho$ is a symmetric free infinitely divisible probability measure. Now from Lemma $8$ of \cite{arizsymmetric} we get
\begin{equation*} 
\mu^{2} = w^{2} \boxtimes \rho \boxtimes \rho  = m \boxtimes \rho \boxtimes \rho  = m \boxtimes \rho_0.
\end{equation*}
Since $\rho_0$ is regularly varying with tail index $-\alpha$ we have from Corollary \ref{theoremhere} that $\mu^2$ is also regularly varying with tail index $-\alpha$. Thus by using the transform $x \mapsto \sqrt{x}$ we get that the symmetric measure $\mu$ is regularly varying with tail index $-\frac{\alpha}{2}$. Thus we have shown \eqref{example1of2}.

The arguments given above can be reversed to show that \eqref{example1of2} implies \eqref{example1of1}.  
\end{proof}
The following is also an immediate consequence of the above discussion and Corollary \ref{corwigprod}.
\begin{corollary}
Let $\alpha \in (0,1)$ and the  measure $\rho_{\frac{2\alpha}{\alpha +1}}$ is free stable of index $\frac{2\alpha}{\alpha +1}$. Then $w \boxtimes \rho_{\frac{2\alpha}{\alpha +1}}$ is regularly varying with tail index $-\frac{\alpha}{2}$.
\end{corollary}

Now we relate our result for the free regular probability measures (Corollary \ref{sita}) and the famous classical result (stated in Theorem \ref{thm: embrechts}) via the notion of Bercovici-Pata bijection.
\begin{definition}[\cite{ber:pata:biane}]\label{ramravan}
The \textbf{Bercovici-Pata bijection} between the set of classical infinitely divisible probability measures $I\left(*\right)$ and the set of free infinitely divisible probability measures $I\left(\boxplus\right)$ is the mapping $\Lambda : I\left(*\right) \rightarrow I\left(\boxplus\right)$ that sends the measure $\mu$ in $I\left(*\right)$ with classical characteristic triplet $\left(\eta, a, \nu\right)$ (see equation \eqref{eq: LK id111}) to the measure $\Lambda\left(\mu\right)$ in $I\left(\boxplus\right)$ with free characteristic triplet $\left(\eta, a, \nu\right)$ (see equation \eqref{lkr}).
\end{definition}
\begin{corollary}\label{classifreeconnec}
Suppose $\alpha \geqslant 0$, $\eta^{'} > 0$ and $\nu \in \mathcal{M}_{+}$ satisfies $\int_{\mathbb{R}^+}min\left(1,t\right)d\nu\left(t\right)<\infty$. Then the classical infinitely divisible probability measure $\nustar$ has regularly varying tail of index $-\alpha$ if and only if the free regular infinitely divisible probability measure $\muminus$, the image of $\nustar$ under Bercovici-Pata bijection, has regularly varying tail of index $-\alpha$. In either case,
$$\nustar\left(x,\infty\right) \sim \muminus\left(x,\infty\right) \text{ as $x\to\infty$.}$$
\end{corollary}

\begin{proof}[Proof of Corollary \ref{classifreeconnec}]
Suppose the classical infinitely divisible probability measure $\nustar$ has regularly varying tail, then by  Theorem~\ref{thm: embrechts} we have the measure $\nu$ in the Laplace transform has the same regularly varying tail. Now both measures in the Laplace transform and the Fourier transform is same $\nu$ by \citep[Remark~21.6]{sato1999ki} since the measure $\nu$ satisfies the conditions $a=0$ in ~\eqref{eq: LK id111}, $\int_{-\infty}^{0}d\nu(t) = 0$, $\int_{0}^{1}td\nu(t) < \infty$ and $\eta > 0$. Then  the relation $\eqref{ram1.1}$ assures that $\sigma$ has also the same regular variation and finally applying corollary~\ref{sita} we can conclude that $\muminus$ has the same regular variation like $\nustar$. The arguments can also be reversed.
\end{proof}

\section{Appendix}\label{sec:appendix}
In the above proofs we have used some important results from \cite{hazra1}. We recall these results here to help the reader. The following two theorems are written in a more compact form which are in particular Theorems $2.1-2.4$ in \cite{hazra1}.
\begin{theorem}\label{thm: error equiv}
Let $p$ be a nonnegative integer and $\mu$ be a probability measure in the class $\mathcal M_p$ and $\alpha\in[p,p+1)$. The following statements are equivalent:
\let\myenumi\theenumi
\renewcommand{\theenumi}{\roman{enumi}}
\begin{enumerate}
\item $y\mapsto \mu\left(y,\infty\right)$ is regularly varying of index $-\alpha$. \label{tail}
\item $y\mapsto \Im r_G\left(iy\right)$ is regularly varying of index $-\left(\alpha-p\right)$. \label{Cauchy remainder}
\item $y\mapsto \Im r_\phi\left(iy\right)$ is regularly varying of index $-\left(\alpha-p\right)$, $\Re r_\phi\left(iy\right)\gg y^{-1}$ as $y\to\infty$ and $r_\phi\left(z\right)\gg z^{-1}$ as $z\to\infty$ n.t. \label{Voiculescu remainder}
\end{enumerate}
If any of the above statements holds, we also have, as $z\to\infty$ n.t., $r_G\left(z\right)\sim r_\phi\left(z\right) \gg z^{-1}$; as $y\to\infty$,
$$\Im r_\phi\left(iy\right) \sim \Im r_G\left(iy\right) \sim -\frac{\frac{\pi\left(p+1-\alpha\right)}2}{\cos\frac{\pi\left(\alpha-p\right)}2} y^p \mu\left(y,\infty\right) \gg \frac1y \text{ and } \Re r_\phi\left(iy\right) \sim \Re r_G\left(iy\right) \gg \frac1y.$$
If $\alpha>p$ and any of the statements~\eqref{tail}-\eqref{Voiculescu remainder} holds, we further have, as $y\to\infty$,
$$\Re r_\phi\left(iy\right) \sim \Re r_G\left(iy\right) \sim -\frac{\frac{\pi\left(p+2-\alpha\right)}2}{\sin\frac{\pi\left(\alpha-p\right)}2} y^p \mu\left(y,\infty\right).$$
If $\alpha=p=0$ and any of the statements~\eqref{tail}-\eqref{Voiculescu remainder} holds, we further have, as $y\to\infty$,
$$\Re r_\phi\left(iy\right) \sim \Re r_G\left(iy\right) \sim - \mu\left(y,\infty\right).$$
\end{theorem}

\begin{theorem}\label{thm: error equiv new}
Let $p$ be a nonnegative integer and $\mu$ be a probability measure in the class $\mathcal M_p$. Let $\beta\in\left(0,1/2\right)$ and $\alpha=p+1$. The following statements are equivalent:
\let\myenumi\theenumi
\renewcommand{\theenumi}{\roman{enumi}}
\begin{enumerate}
\item $y\mapsto \mu\left(y,\infty\right)$ is regularly varying of index $-\left(p+1\right)$. \label{tail new}
\item $y\mapsto \Re r_G\left(iy\right)$ is regularly varying of index $-1$. \label{Cauchy remainder new}
\item $y\mapsto \Re r_\phi\left(iy\right)$ is regularly varying of index $-1$, $y^{-1} \ll \Im r_\phi\left(iy\right)\ll y^{-\left(1-\beta/2\right)}$ as $y\to\infty$ and $z^{-1} \ll r_\phi\left(z\right) \ll z^{-\beta}$ as $z\to\infty$ n.t. \label{Voiculescu remainder new}
\end{enumerate}
If any of the above statements holds, we also have, as $z\to\infty$ n.t., $z^{-1}\ll r_G\left(z\right)\sim r_\phi\left(z\right)\ll z^{-\beta}$; as $y\to\infty$,
$$y^{-\left(1+\beta/2\right)} \ll \Re r_\phi\left(iy\right) \sim \Re r_G\left(iy\right) \sim -\frac\pi2 y^p \mu\left(y,\infty\right) \ll y^{-\left(1-\beta/2\right)}$$ and
$$y^{-1} \ll \Im r_\phi\left(iy\right) \sim \Im r_G\left(iy\right) \ll y^{-\left(1-\beta/2\right)}.$$
\end{theorem}

\bibliographystyle{abbrvnat}
\bibliography{infdivbib}
\end{document}